\documentclass[12pt]{amsart}

\usepackage{amssymb,epsfig}
\usepackage{amsmath,verbatim,psfrag}
\usepackage[all,2cell,dvips]{xy}

\newcommand{\nc}{\newcommand}
\nc{\nt}{\newtheorem}
\nc{\bs}{\bigskip}
\nc{\dmo}{\DeclareMathOperator}

\nt{main}{Main Theorem}
\nt{thm}{Theorem}[section]
\nt{prop}[thm]{Proposition}
\nt{lem}[thm]{Lemma}
\nt{fact}[thm]{Fact}
\nt{cor}[thm]{Corollary}
\nt{conj}[thm]{Conjecture}
\nt{mainconj}{Conjecture}
\nt{question}[thm]{Question}
\nt{problem}[thm]{Problem}

\nc{\I}{\mathcal{I}}
\nc{\Br}{\mathcal{H}}
\nc{\T}{\mathcal{T}}
\nc{\K}{\mathcal{K}}
\nc{\SI}{\mathcal{SI}}
\nc{\SL}{\mathcal{SL}}
\nc{\SK}{\mathcal{SK}}
\nc{\ST}{\mathcal{ST}}
\nc{\SIBK}{\mathcal{SIBK}}
\nc{\SBK}{\mathcal{SBK}}
\nc{\Push}{\mathcal Push}
\dmo{\Mod}{Mod}
\dmo{\PMod}{PMod}
\dmo{\SMod}{SMod}
\dmo{\SHomeo}{SHomeo}
\dmo{\Homeo}{Homeo}
\dmo{\Sp}{Sp}
\dmo{\SSp}{SSp}
\dmo{\Stab}{Stab}
\dmo{\Isom}{Isom}
\dmo{\Comm}{Comm}
\dmo{\B}{B}
\dmo{\PB}{PB}

\nc{\Z}{\mathbb{Z}}
\nc{\R}{\mathbb{R}}

\nc{\p}[1]{\medskip\paragraph{{\bf #1}}}
\nc{\margin}[1]{\marginpar{\scriptsize #1}}

\nc{\bl}{ \begin{list}{$\cdot$}{
\setlength{\leftmargin}{.5in}
\setlength{\rightmargin}{.5in}
\setlength{\parsep}{0.5ex plus .2ex minus 0ex}
\setlength{\itemsep}{0.2ex plus 0.2ex minus 0ex}
}
}

\nc{\el}{\end{list}}

\title{Point pushing, homology, and the hyperelliptic involution}

\begin{document}
	
\input{epsf.sty}

\author{Tara E. Brendle}

\author{Dan Margalit}

\address{Tara E. Brendle \\ School of Mathematics \& Statistics \\ University Gardens \\ University of Glasgow \\ G12 8QW \\ tara.brendle@glasgow.ac.uk}

\address{Dan Margalit \\ School of Mathematics\\ Georgia Institute of Technology \\ 686 Cherry St. \\ Atlanta, GA 30332 \\  margalit@math.gatech.edu}

\thanks{The second author gratefully acknowledges support from the Sloan Foundation and the National Science Foundation.}

\keywords{Torelli group, Johnson kernel, hyperelliptic mapping class group, Birman exact sequence, Burau representation}

\subjclass[2000]{Primary: 20F36; Secondary: 57M07}

\begin{abstract}
The hyperelliptic Torelli group is the subgroup of the mapping class group consisting of elements that act trivially on the homology of the surface and that also commute with some fixed hyperelliptic involution.  We prove a Birman exact sequence for hyperelliptic Torelli groups, and we show that this sequence splits.  As a consequence, we show that the hyperelliptic Torelli group is generated by Dehn twists if and only if it is generated by reducible elements.  We also give an application to the kernel of the Burau representation.
\end{abstract}

\maketitle

\section{Introduction}

Let $S_g$ denote a closed, connected, orientable surface of genus $g$.  The hyperelliptic Torelli group $\SI(S_g)$ is the subgroup of the mapping class group $\Mod(S_g)$ consisting of all elements that act trivially on $H_1(S_g;\Z)$ and that commute with the isotopy class of  some fixed hyperelliptic involution $s : S_g \to S_g$ (see Figure~\ref{figure:hi} below).

The group $\SI(S_g)$ arises in algebraic geometry in the following context.  Let $\T(S_g)$ denote the cover of the moduli space of Riemann surfaces corresponding to the Torelli subgroup $\I(S_g)$ of $\Mod(S_g)$.  This is the subgroup of $\Mod(S_g)$ consisting of all elements acting trivially on $H_1(S_g;\Z)$.  The period mapping is a function on $\T(S_g)$ whose image lies in the Siegel upper half-space of rank $g$.  This map is branched over the subset $\Br(S_g)$ of $\T(S_g)$ that is fixed by the action of $s$ and
\[ \pi_1(\Br(S_g)) \cong \SI(S_g). \]
Because of this, $\SI(S_g)$ is related, for example, to the topological Schottky problem; see \cite[Problem 1]{hain}.

A basic tool in the theory of mapping class groups is the Birman exact sequence.  This sequence relates the mapping class group of a surface with marked points to the mapping class group of the surface obtained by forgetting the marked points; see Section~\ref{section:bes}.  This is a key ingredient for performing inductive arguments on the mapping class group.  For instance, the standard proof that $\Mod(S_g)$ is generated by Dehn twists uses the Birman exact sequence in the inductive step on genus.

The main goal of this paper is to provide a Birman exact sequence for $\SI(S_g)$.  As in the case of $\Mod(S_g)$, the Birman exact sequence is crucial for inductive arguments.  For example, the authors and Childers have recently used our Birman exact sequences in order to show that the top-dimensional homology of $\SI(S_g)$ is infinitely generated \cite{cdsi}.

\p{Main theorems} In order to state our main theorems, we need to define the hyperelliptic Torelli group for a marked surface.  The hyperelliptic Torelli group $\SI(S_g,P)$ of the surface $S_g$ with marked points $P$ is the group of isotopy classes of homeomorphisms of $S_g$ that commute with $s$, that preserve the set $P$, and that act trivially on the homology of $S_g$ relative to $P$.

There is a forgetful homomorphism $\SI(S_g,P) \to \SI(S_g)$, and our Birman exact sequences give a precise description of the kernel in two cases.  In the first case, we show that the kernel is trivial, and so the Birman exact sequence degenerates to an isomorphism.

\begin{thm}
\label{thm:sibes1}
Let $g \geq 0$ and let $P$ be a single point in $S_g$ fixed by $s$.  The forgetful map $\SI(S_g,P) \to \SI(S_g)$ is an isomorphism.
\end{thm}

Let $S_g^1$ denote a surface of genus $g$ with one boundary component.  We prove in Theorem~\ref{thm:uncapping} that
\[ \SI(S_g^1) \cong \SI(S_g) \times \Z. \]
It is surprising to realize $\SI(S_g)$ as a subgroup of $\SI(S_g^1)$ since there is no inclusion $S_g \to S_g^1$.

Next we consider the case where $P$ is a pair of distinct points interchanged by $s$.  The Birman--Hilden theorem (Theorem~\ref{thm:sbes2}) identifies the kernel of $\SI(S_g,P) \to \SI(S_g)$ as a subgroup of $F_{2g+1} \cong \pi_1(S_{0,2g+2})$, where $S_{0,2g+2}$ is a sphere with $2g+2$ punctures.  Denote the generators of $F_{2g+1}$ by $\zeta_1,\dots,\zeta_{2g+1}$ and the generators for the group $\Z^{2g+1}$ by $e_1, \dots, e_{2g+1}$.  Denote by $F_{2g+1}^{even}$ the subgroup of $F_{2g+1}$ consisting of all even length words in the $\zeta_i$.  There is a homomorphism $\epsilon : F_{2g+1}^{even} \to \Z^{2g+1}$ defined by 
\[ \zeta_{i_1}^{\alpha_1} \zeta_{i_2}^{\alpha_2} \mapsto e_{i_1} - e_{i_2} \]
where $\alpha_j = \pm 1$ for each $j$.

\begin{thm}
\label{thm:algchar}
Let $g \geq 1$, and let $P$ be a pair of distinct points in $S_g$ interchanged by $s$.  If we identify the kernel of the map $\SI(S_g,P) \to \SI(S_g)$ with a subgroup of $F_{2g+1}$ as above, then the following sequence is split exact:
\[ 1 \to \ker \epsilon \to \SI(S_g,P) \to \SI(S_g) \to 1. \]
\end{thm}
Again, the fact that the short exact sequence in Theorem~\ref{thm:algchar} is split is unexpected because there is no splitting induced by a map $S_g \to S_g-P$.

Since $\epsilon$ is a map from a nonabelian free group onto an infinite abelian group, its kernel is an infinitely generated free group.  We thus obtain the following.

\begin{cor}
\label{cor:surprise}
Let $g \geq 1$, and let $P$ be a pair of distinct points in $S_g$ interchanged by $s$.  We have
\[ \SI(S_g, P) \cong \SI(S_g) \ltimes F_\infty. \]
\end{cor}

\p{Application to generating sets} A simple closed curve in $S_g$ is \emph{symmetric} if it is fixed by the hyperelliptic involution $s$.  We prove in Theorem~\ref{thm:topchar} that the image of each element of $\ker \epsilon$ in $\SI(S_g,P)$ is a product of Dehn twists about symmetric separating curves.

Hain has conjectured that the entire group $\SI(S_g)$ is in fact generated by Dehn twists about symmetric separating curves \cite[Conjecture 1]{hain}; see also Morifuji \cite[Section 4]{morifuji}.  Hain's conjecture is well-known to be true for $g=2$; see Theorem~\ref{thm:g21} below.  Using Theorem~\ref{thm:topchar}, we prove the following theorem in Section~\ref{sec:app}.  In the statement, an element of $\Mod(S_g)$ is \emph{reducible} if it fixes a collection of isotopy classes of pairwise disjoint simple closed curves in $S_g$.

\begin{thm}
\label{thm:reducibles}
Let $g \geq 0$.  Suppose that $\SI(S_k)$ is generated by reducible elements for all $k$ between 0 and $g$, inclusive.
Then the group $\SI(S_g)$ is generated by Dehn twists about symmetric separating curves.
\end{thm}

In other words, by the results of this paper, Hain's conjecture is reduced to showing that $\SI(S_g)$ is generated by reducible elements.

Also, we prove in Theorem~\ref{thm:g21} that the hyperelliptic Torelli group for a surface of genus two with one marked point or one boundary component is generated by Dehn twists about symmetric separating curves.

\p{Application to the Burau representation} In the proof of Theorem~\ref{thm:uncapping}, we will explain how to identify $\SI(S_g^1)$ with a subgroup of the pure braid group $PB_{2g+1}$.  It is then easy to check that $\SI(S_g^1)$ is isomorphic to the kernel $K_{2g+1}$ of the reduced Burau representation of $PB_{2g+1}$ evaluated at $t=-1$; see \cite[Remark 4.3]{perron}, \cite{mp}, and \cite{companion}.  Using the fact that $\SI(S_g^1)$ splits over its center (Theorem~\ref{thm:uncapping}), it follows that $K_{2g+1}$ also splits as a direct product over its center $Z(K_{2g+1})$ and $K_{2g+1}/Z(K_{2g+1}) \cong \SI(S_g)$.

We can also use the Birman exact sequence in Theorem~\ref{thm:algchar} to relate $K_{2g+2}$ to $K_{2g+1}$.  Analogously to the case of odd index braid groups, we have $\SI(S_g,P) \cong K_{2g+2}/Z(K_{2g+2})$.  In the even degree case, we also have $Z(K_{2g+2}) = Z(B_{2g+2})$, where $B_{2g+2}$ is the braid group on $2g+2$ strands.  Thus, by Corollary~\ref{cor:surprise}, we have
\[ K_{2g+2}/Z(B_{2g+2}) \cong (K_{2g+1}/Z(K_{2g+1})) \ltimes F_{\infty}. \]

\p{Acknowledgments} We would also like to thank Joan Birman, Kai-Uwe Bux, Tom Church, Benson Farb, Dick Hain, Chris Leininger, and Andy Putman for helpful discussions.

%%%
%%%
%%%

\section{Hyperelliptic mapping class groups, hyperelliptic Torelli groups, and the Birman--Hilden theorem}
\label{sec:bg}

We recall some basic information about hyperelliptic mapping class groups, including the fundamental theorem of Birman--Hilden.  Theorems~\ref{thm:bh}, \ref{thm:bh marked}, and~\ref{thm:bh low genus} below are all special cases of their general theorem \cite[Theorem 1]{birmanhilden}.

\p{Mapping class groups} Let $S$ be a compact, connected, orientable surface with finitely many marked points in its interior.  The \emph{mapping class group} $\Mod(S)$ is the group of homotopy classes of homeomorphisms of $S$, where all homeomorphisms and homotopies are required to fix the marked points as a set and to fix $\partial S$ pointwise.  

\p{Hyperelliptic involutions and mapping class groups} A hyperelliptic involution is an order two homeomorphism of $S_g$ that acts by $-I$ on $H_1(S_g;\Z)$.  We fix some hyperelliptic involution $s$ of $S_g$, once and for all.  Its mapping class, which we also call a hyperelliptic involution, will be denoted $\sigma$.  The mapping class $\sigma$ is unique up to conjugacy.  

\begin{figure}[htb]
\psfrag{...}{$\dots$}
\centerline{\includegraphics[scale=.5]{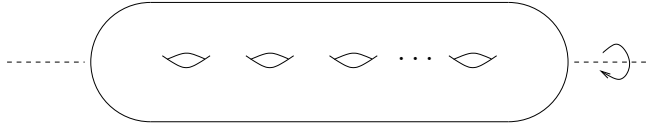}}
\caption{Rotation by $\pi$ about the indicated axis is a hyperelliptic involution.}
\label{figure:hi}
\end{figure}

The \emph{hyperelliptic mapping class group}, or \emph{hyperelliptic mapping class group}, is the group $\SMod(S_g)$ of isotopy classes of orientation-preserving homeomorphisms of $S_g$ that commute with $s$ (isotopies are not required to be $s$-equivariant).  The \emph{hyperelliptic Torelli group} of $S_g$ is the group
\[ \SI(S_g) = \I(S_g) \cap \SMod(S_g). \]

Suppose that $P$ is a set of marked points in $S_g$, and denote the marked surface by $(S_g,P)$.  The hyperelliptic mapping class group $\SMod(S_g,P)$ is the subgroup of $\Mod(S_g,P)$ consisting of elements represented by homeomorphisms that commute with $s$.  The hyperelliptic Torelli group $\SI(S_g,P)$ is the subgroup of $\SMod(S_g,P)$ consisting of elements that act trivially on the relative homology $H_1(S_g,P;\Z)$.

\p{Hyperelliptic mapping class groups in low genus} Let $g \in \{0,1,2\}$.  The group $\Mod(S_g)$ has a generating set consisting of Dehn twists about symmetric simple closed curves.  Each such Dehn twist has a representative that commutes with $s$, and so we obtain the following.

\begin{fact}
\label{fact:low genus}
For $g \in \{0,1,2\}$, we have
\[ \SMod(S_g) = \Mod(S_g). \]
\end{fact}

\p{The Birman--Hilden theorem} For $g \geq 3$, the group $\SMod(S_g)$ has infinite index in $\Mod(S_g)$.  Indeed, if $a$ is any isotopy class of simple closed curves in $S_g$ that is not fixed by $\sigma$, then no nontrivial power of the Dehn twist $T_a$ is an element of $\SMod(S_g)$ (note that, by Fact~\ref{fact:low genus} no such curves exist in a closed genus two surface!).  However, there is a very useful description of $\SMod(S_g)$ given by Birman--Hilden, which we now explain.

The quotient of $S_g$ by $s$ is a sphere with $2g+2$ marked points, namely the images of the fixed points of $s$.  We denote this surface by $S_{0,2g+2}$.  Any homeomorphism of $S_g$ that commutes with $s$ necessarily fixes the fixed points of $s$, and hence descends to a homeomorphism of $S_{0,2g+2}$ preserving the marked points.  

By definition, any element $f$ of $\SMod(S_g)$ has a representative $\phi$ that commutes with $s$.  Thus, there is a map
\[ \theta : \SMod(S_g) \to \Mod(S_{0,2g+2}). \]
The following is a special case of a theorem of Birman--Hilden.

\begin{thm}[Birman--Hilden]
\label{thm:bh}
For $g \geq 2$, the map $\theta : \SMod(S_g) \to \Mod(S_{0,2g+2})$ is a well-defined, surjective homomorphism with kernel $\langle \sigma \rangle$.  In particular, $\SMod(S_g)/\langle \sigma \rangle \cong \Mod(S_{0,2g+2})$.
\end{thm}

\p{A Birman--Hilden theorem for surfaces with marked points} Let $p_1,p_2 \in S_g$ be distinct points that are interchanged by $s$.  The quotient $(S_g,\{p_1,p_2\})/\langle s \rangle$ is the pair $(S_{0,2g+2},p)$, where $p \in S_{0,2g+2}$ is the image of $p_1 \cup p_2$.  
Elements of $\Mod(S_{0,2g+2},p)$ can permute the $2g+2$ marked points coming from $S_{0,2g+2}$, but must preserve the marked point $p$.  As before there is a map $\theta : \SMod(S_g,\{p_1,p_2\}) \to \Mod(S_{0,2g+2},p)$.  We have the following analogue of Theorem~\ref{thm:bh}.

\begin{thm}[Birman--Hilden]
\label{thm:bh marked}
For $g \geq 1$, the homomorphism $\theta : \SMod(S_g,\{p_1,p_2\}) \to \Mod(S_{0,2g+2},p)$ is a well-defined, surjective homomorphism with kernel $\langle \sigma \rangle$.  In particular, $\SMod(S_g,\{p_1,p_2\})/\langle \sigma \rangle \cong \Mod(S_{0,2g+2},p)$.
\end{thm}

\p{The Birman--Hilden theorem in low genus} Theorem~\ref{thm:bh} does not hold as stated for $g \in \{0,1\}$.  In fact, in these cases, the map $\theta$ is not even well defined.  This is because there are nontrivial finite order homeomorphisms of $S_{0,2g+2}$ that lift to homeomorphisms of $S_g$ that are homotopic to the identity.  Therefore, we are forced to redefine $\theta$ in these cases.  Let $p$ be some particular fixed point of ${s}$.  We have
\[ \SMod(S^2,p) = \Mod(S^2,p) = \Mod(S^2) = \SMod(S_2) = 1 \] and
\[ \SMod(T^2,p) = \Mod(T^2,p) \cong \Mod(T^2) = \SMod(T^2) \cong \textrm{SL}(2,\Z). \]
 Therefore, for $g \in \{0,1\}$, we can define $\theta$ via the composition
\[ \SMod(S_g) \stackrel{\cong}{\to} \SMod(S_g,p) \to \Mod(S_{0,2g+2}). \]
For $g=0$, this map $\theta$ is the trivial map $\Mod(S^2) \to \Mod(S_{0,2}) \cong \Z/2\Z$.

\begin{thm}[Birman--Hilden]
\label{thm:bh low genus}
The map $\theta : \SMod(T^2) \to \Mod(S_{0,4})$ is a well-defined homomorphism with kernel $\langle \sigma \rangle$.  Its image is the subgroup of $\Mod(S_{0,4})$ consisting of elements that fix the marked point corresponding to $p \in T^2$.
\end{thm}

\section{Birman exact sequences for the hyperelliptic mapping class group}
\label{section:bes}

In this section we give Birman exact sequences for hyperelliptic mapping class groups in the two cases which will be of interest for us: first, forgetting one marked point, and then forgetting two.  

We begin by recalling the classical Birman exact sequence.  Let $S$ denote a connected, orientable, compact surface with finitely many marked points in its interior.  Assume that the surface obtained by removing the marked points from $S$ has negative Euler characteristic.   Let $p \in S$ be a marked point (distinct from any others in $S$).  There is a forgetful map $\Mod(S,p) \to \Mod(S)$, and the Birman exact sequence identifies the kernel of this map with $\pi_1(S,p)$:
\[ 1 \to \pi_1(S,p) \stackrel{\Push}{\to} \Mod(S,p) \to \Mod(S) \to 1.\]
Given an element $\alpha$ of $\pi_1(S,p)$, we can describe $\Push(\alpha)$ as the map obtained by pushing $p$ along $\alpha$; see \cite[Section 5.2]{primer} or \cite[Section 1]{birmanes}.

\subsection{Forgetting one point}

Let $p \in S_g$ be a fixed point of ${s}$.  As in the classical Birman exact sequence, there is a forgetful map $\SMod(S_g, p) \to \SMod(S_g)$.  

\begin{thm}
\label{thm:sbes1}
Let $g \geq 0$.  The forgetful map $\SMod(S_g, p) \to \SMod(S_g)$ is injective.
\end{thm}

Note that this map is not surjective for $g \geq 2$.  For example, its image does not contain a Dehn twist about a symmetric curve through $p$. 

\begin{proof}

The group $\SMod(S^2,p)$ is trivial, so there is nothing to show in this case.  For $g=1$, we can use Fact~\ref{fact:low genus} plus the fact that the forgetful map $\Mod(T^2,p) \to \Mod(T^2)$ is an isomorphism \cite[Section 5.2]{primer}.  So assume $g \geq 2$.  

The classical Birman exact sequence gives:
\[ 1 \to \pi_1(S_g) \stackrel{\Push}{\to} \Mod(S_g,p) \to \Mod(S_g) \to 1 .\]
Therefore, to prove the theorem, we need to show that the image of $\pi_1(S_g)$ in $\Mod(S_g,p)$ intersects $\SMod(S_g,p)$ trivially.  In other words, we need to show that $\sigma$ does not commute with any nontrivial element of the image of $\pi_1(S_g)$.  For $f \in \Mod(S_g,p)$ and $\alpha \in \pi_1(S_g)$, we have that $f \Push(\alpha) f^{-1} = \Push(f_\star(\alpha))$.  Therefore, we need to show that $\sigma_\star(\alpha) \neq \alpha$ for all nontrivial $\alpha \in \pi_1(S_g)$.

Choose a hyperbolic metric on $S_g$ so that $s$ is an isometry.  A concrete way to do this is to identify $S_g$ with a hyperbolic $(4g+2)$-gon with opposite sides glued, and take $s$ to be rotation by $\pi$ through the center.

Next, choose a universal metric covering $\mathbb{H}^2 \to S_{g}$.  The preimage of $p$ in $\mathbb{H}^2$ is the set $\{ \gamma \cdot \widetilde p : \gamma \in \pi_1(S_g) \}$, where $\widetilde p$ is some fixed lift of $p$.

The map ${s}$ has a unique lift $\widetilde {s}$ to $\Isom^+(\mathbb{H}^2)$ that fixes $\widetilde p$.  This lift has order two.  By the classification of elements of $\Isom^+(\mathbb{H}^2)$, it is a rotation by $\pi$.  Thus, $\widetilde {s}$ has exactly one fixed point.

The action of $\widetilde {s}$ on the set $\{ \gamma \cdot \widetilde p\}$ is given by
\[ \gamma \cdot \widetilde p \mapsto {s}_\star(\gamma) \cdot \widetilde p. \]
If $\sigma_\star(\alpha)=\alpha$, then it follows that $\widetilde {s}$ fixes $\alpha \cdot \widetilde p$.  But we already said that $\widetilde {s}$ has a unique fixed point, namely $\widetilde p$.  So $\alpha=1$, as desired.
\end{proof}

\subsection{Forgetting two points} 
\label{subsection:forget2}

Let $p_1,p_2 \in S_g$, with ${s}(p_1)=p_2$, and let $p$ denote the image of $p_1 \cup p_2$ in the quotient sphere $S_{0,2g+2}$.  Let $\SBK(S_g,\{p_1,p_2\})$ denote the kernel of the forgetful homomorphism $\SMod(S_g, \{ p_1 , p_2 \}) \to \SMod(S_g)$ (the notation is for ``symmetric Birman kernel'').  We have a short exact sequence:
\[ 1 \to \SBK(S_g,\{p_1,p_2\}) \to \SMod(S_g, \{ p_1 , p_2 \}) \to \SMod(S_g) \to 1.\]

By $\pi_1(S_{0,2g+2},p)$ we mean the fundamental group of the complement in $S_{0,2g+2}$ of the $2g+2$ marked points.

\begin{thm}
\label{thm:sbes2}
Let $g \geq 1$.  We have that $\SBK(S_g,\{p_1,p_2\}) \cong F_{2g+1}$, where $F_{2g+1}$ is identified with $\pi_1(S_{0,2g+2},p)$.
\end{thm}

\begin{proof}

We have the following commutative diagram.
\[
\xymatrix{
& & 1\ar[d] & 1 \ar[d]&\\
& & \langle \sigma \rangle  \ar[d] \ar[r]^\cong & \langle \sigma \rangle \ar[d]  & \\
1 \ar[r]& \SBK(S_g,\{p_1,p_2\})  \ar[r]& \SMod(S_g, \{p_1, p_2 \}) \ar[d] \ar[r] & \SMod(S_g) \ar[r] \ar[d] & 1 \\
1 \ar[r]& \pi_1(S_{0,2g+2},p) \ar@{}[d]^{\rotatebox[origin=c]{270}{$\cong$}}
 \ar[r] & \Mod(S_{0,2g+2},p) \ar[r] \ar[d] & \Mod(S_{0,2g+2}) \ar[r] & 1\\
&\ \ \ F_{2g+1} & 1 & & 
}
\]
The second horizontal short exact sequence is an instance of the Birman exact sequence, and the two vertical sequences are given by Theorems~\ref{thm:bh}, \ref{thm:bh marked}, and~\ref{thm:bh low genus}.  From the diagram it is straightforward to see that $\SBK(S_g,\{p_1,p_2\}) \cong \pi_1(S_{0,2g+2})$.
\end{proof}

\section{Birman exact sequences for hyperelliptic Torelli groups: main results}
\label{sec:main}

The main results of the paper are Birman exact sequences for hyperelliptic Torelli groups.  As in the previous section, there are two versions, corresponding to forgetting one point (Theorem~\ref{thm:sibes1}) and forgetting two points (Theorem~\ref{thm:algchar}).

\subsection{Forgetting one point} 

Let $\PMod(S_{0,2g+2})$ denote the subgroup of $\Mod(S_{0,2g+2})$ consisting of elements that induce the trivial permutation of the marked points.  The next fact is an easy exercise (it also follows immediately from the main result of \cite{arnold}).

\begin{lem}
\label{lem:torelli pure}
Let $g \geq 0$.  Under the map $\theta:\SMod(S_g) \to \Mod(S_{0,2g+2})$, the image of $\SI(S_g)$ lies in $\PMod(S_{0,2g+2})$.  
\end{lem}

We are now ready for the proof of our first Birman exact sequence for hyperelliptic Torelli groups.

\begin{proof}[Proof of Theorem~\ref{thm:sibes1}]

It follows immediately from Theorem~\ref{thm:sbes1} that the homomorphism $\SI(S_g,p) \to \SI(S_g)$ is injective.  We will show that it is also surjective.  Let $f \in \SI(S_g)$ and let $\phi$ be a representative homeomorphism.  By Lemma~\ref{lem:torelli pure}, the induced homeomorphism of $S_{0,2g+2}$ fixes each of the $2g+2$ marked points.  It follows that $\phi$ fixes $p$ and that $\phi$ represents an element $\widetilde f$ of $\SI(S_g,p)$ that maps to $f$.
\end{proof}

Before moving on to the second Birman exact sequence for the hyperelliptic Torelli group, we give a variation of Theorem~\ref{thm:sibes1}, where we forget a boundary component (really, cap a boundary component) instead of a marked point.

\p{Capping a boundary component} Let $p \in S_g$ be a fixed point of ${s}$, and let $\Delta \subset S_g$ be an embedded disk that contains $p$ and is fixed by ${s}$.  The surface $S_g-\Delta$ is a compact surface of genus $g$ with one boundary component, which we denote by $S_g^1$.

The hyperelliptic mapping class group $\SMod(S_g^1)$ is defined as the group of isotopy classes of orientation-preserving homeomorphisms that commute with $s$ are restrict to the identity on $\partial S_g^1$.  Also, $\SI(S_g^1)$ is the subgroup of $\SMod(S_g^1)$ consisting of all elements that act trivially on $H_1(S_g^1;\Z)$.  The Dehn twist $T_{\partial S_g^1}$ is an infinite order element of $\SI(S_g^1)$.

The inclusion $S_g^1 \to S_g$ induces a homomorphism $\SI(S_g^1) \to \SI(S_g)$, and so we can again ask about the kernel.  In this case we have the following.

\begin{thm}
\label{thm:uncapping}
Let $g \geq 1$.  We have
\[ \SI(S_g^1) \cong \SI(S_g) \times \Z, \]
where the map $\SI(S_g^1) \to \SI(S_g)$ is the one induced by the inclusion $S_g^1 \to S_g$, and where the $\Z$ factor is $\langle T_{\partial S_g^1} \rangle$.  
\end{thm}

\begin{proof}

One version of the Birman--Hilden theorem \cite[Theorem 9.2]{primer} identifies $\SMod(S_g^1)$ isomorphically with the mapping class group of a disk $D_{2g+1}$ with $2g+1$ marked points.  The latter is isomorphic to the braid group $B_{2g+1}$ on $2g+1$ strands.

By Lemma~\ref{lem:torelli pure}, the group $\SI(S_g^1)$ is identified isomorphically with a subgroup of $PB_{2g+1}$, the subgroup of $B_{2g+1}$ consisting of elements that fix each marked point/strand.

For any $n$, the group $PB_n$ splits as a direct product over its center, which is generated by the Dehn twist $T_{\partial D_n}$ \cite[Section 9.3]{primer}.  Under the restriction $\bar \theta : \SI(S_g) \hookrightarrow PB_{2g+1}$, we have $\bar \theta^{-1}(Z(PB_{2g+1})) = \langle T_{\partial S_g^1} \rangle$.  Thus, $\SI(S_g^1)$ splits as a direct product over $\langle T_{\partial S_g^1} \rangle$.

It remains to show that $\SI(S_g^1)/ \langle T_{\partial S_g^1} \rangle \cong \SI(S_g)$.  There is a short exact sequence
\[
1 \to \langle T_{\partial S_g^1} \rangle \to \Mod(S_g^1) \to \Mod(S_g,p) \to 1,
\]
where the map $\Mod(S_g^1) \to \Mod(S_g,p)$ is the one induced by the inclusion $S_g^1 \to (S_g,p)$; see \cite[Proposition 3.19]{primer}.  On the level of hyperelliptic Torelli groups, this gives
\[ 1 \to \langle T_{\partial S_g^1} \rangle \to \SI(S_g^1) \to \SI(S_g,p) \to 1. \]
We have already shown that $\SI(S_g,p) \cong \SI(S_g)$ (Theorem~\ref{thm:sibes1}).  Thus, $\SI(S_g^1)/ \langle T_{\partial S_g^1} \rangle \cong \SI(S_g)$, and we are done.
\end{proof}

In the proof of Proposition~\ref{prop:g12} we will need a version of Theorem~\ref{thm:uncapping} for a surface with two marked points.  Let $g \geq 1$, and let $p_1$ and $p_2$ be distinct points in $S_g^1$ interchanged by $s$.

We define $\SMod(S_g^1,\{p_1,p_2\})$ in the usual way, and then we define $\SI(S_g^1,\{p_1,p_2\})$ as the kernel of the action of $\SMod(S_g^1,\{p_1,p_2\})$ on $H_1(S_g^1,\{p_1,p_2\};\Z)$.

By essentially the same argument as the one used for Theorem~\ref{thm:uncapping}, we have, for $g \geq 0$,
\[ \SI(S_g^1,\{p_1,p_2\}) \cong \SI(S_g,\{p_1,p_2\}) \times \Z, \]
where the $\Z$ factor is the Dehn twist about $\partial S_g^1$.

%%%
%%%
%%%

\subsection{Forgetting two points}
\label{siforget2}

Recall from Theorem~\ref{thm:sbes2} that the kernel of $\SMod(S_g, \{ p_1 , p_2 \}) \to \SMod(S_g)$ is $\SBK(S_g,\{p_1,p_2\}) \cong F_{2g+1}$, which is identified with $\pi_1(S_{0,2g+2})$.  Let $p$ denote the image of $p_1 \cup p_2$ in $S_{0,2g+2}$. Let $\zeta_1,\dots,\zeta_{2g+1}$ be the generators for $\pi_1(S_{0,2g+2},p) \cong F_{2g+1}$ shown in Figure~\ref{figure:zetas}.  In what follows, we identify $F_{2g+1}$ with $\pi_1(S_{0,2g+2},p)=\langle \zeta_i\rangle$.

\begin{figure}[htb]
\psfrag{...}{$\dots$}
\psfrag{z1}{$\zeta_{2g+1}$}
\psfrag{z2}{$\zeta_{2g}$}
\psfrag{z2g-1}{$\zeta_1$}
\psfrag{abar}{$p$}
\centerline{\includegraphics[scale=1.25]{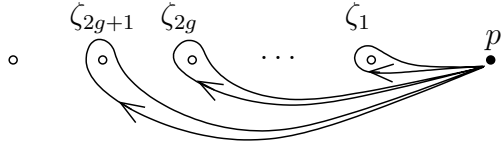}}
\caption{The elements $\zeta_i$ of $\pi_1(S_{0,2g+2},p)$.}
\label{figure:zetas}
\end{figure}

Let  $\SIBK(S_g,\{p_1,p_2\})$ denote the kernel of the forgetful homomorphism $\SI(S_g, \{ p_1,p_2 \}) \to \SI(S_g)$, that is:
\[ 1 \to \SIBK(S_g,\{p_1,p_2\}) \to \SI(S_g, \{ p_1,p_2 \}) \to \SI(S_g) \to 1\]
(the notation stands for ``symmetric Torelli Birman kernel'').  

Consider the homomorphism $F_{2g+1} \to \Z$ obtained by sending each $\zeta_i$ to 1.  As in the introduction, we define the \emph{even subgroup} of $F_{2g+1}$ to be the preimage of $2\Z$, and we denote it by $F_{2g+1}^{even}$.  The elements of $F_{2g+1}^{even}$ are products $\zeta_{i_1}^{\alpha_1}\zeta_{i_2}^{\alpha_2}\cdots \zeta_{i_k}^{\alpha_k}$ where $k$ is even and $\alpha_i \in \{-1,1\}$ for all ${i}$.  

As in the introduction, let 
\[ \epsilon : F_{2g+1}^{even} \to \Z^{2g+1} \]
be the homomorphism given by
\[ \zeta_{i_1}^{\alpha_1}\zeta_{i_2}^{\alpha_2}\cdots \zeta_{i_k}^{\alpha_k} \mapsto \sum_{j=1}^k (-1)^{j+1} e_{i_k} \]
where $\{e_1, \ldots, e_{2g+1}\}$ are the standard generators for $\Z^{2g+1}$.

Theorem~\ref{thm:algchar} states that, as a subgroup of $\SBK(S_g,\{p_1,p_2\})$, the group $\SIBK(S_g,\{p_1,p_2\})$ is equal to the image of $\ker \epsilon$ under the isomorphism
\[  \langle \zeta_1,\dots,\zeta_{2g+1} \rangle = F_{2g+1} \cong \pi_1(S_{0,2g+2},p) \stackrel{\cong}{\to} \SBK(S_g,\{p_1,p_2\}).  \]

In order to prove Theorem~\ref{thm:algchar}, we will need two lemmas describing the action of elements of $\SBK(S_g,\{p_1,p_2\})$ on the relative homology $H_1(S_g,\{p_1,p_2\};\Z)$.  Our argument has its origins in the work of Johnson \cite[Section 2]{dj2}, van den Berg \cite[Section 2.4]{vdb}, and Putman \cite[Section 4]{cutpaste}.

A \emph{proper arc} $\alpha$ in a surface $S$ with marked points $\{p_i\}$ is a map $\alpha : I \to (S,\{p_i\})$ where $\alpha^{-1}(\{p_i\}) = \{0,1\}$.

\begin{lem}
\label{lem:fix2}
Let $g \geq 1$. If $f$ is an element of $\SBK(S_g,\{p_1,p_2\})$, and if $\beta$ is any oriented proper arc in $(S_g,\{p_1.p_2\})$ connecting the two marked points, then $f$ is an element of $\SIBK(S_g,\{p_1,p_2\})$ if and only if in $H_1(S_g,\{p_1,p_2\};\Z)$ we have  $f_\star([\beta]) = [\beta]$.
 \end{lem}

\begin{proof}

One direction is trivial: if $f \in \SIBK(S_g,\{p_1,p_2\})$, then by definition, $f$ acts trivially on $H_1(S_g,\{p_1,p_2\};\Z)$.

We now prove the other direction.  There is a basis for $H_1(S_g,\{p_1,p_2\};\Z)$ given by (the classes of) finitely many oriented closed curves plus the oriented arc $\beta$.
Thus, to prove the lemma, we only need to show that any $f \in \SBK(S_g,\{p_1,p_2\})$ preserves the class in $H_1(S_g,\{p_1,p_2\};\Z)$ of each oriented closed curve in $S_g$.

Let $\phi$ be a representative of $f$.  We can regard $\phi$ either as a homeomorphism of $(S_g,\{p_1,p_2\})$ or as a homeomorphism of $S_g$.  Also, let $\gamma$ be an oriented closed curve in $S_g$.  We can similarly regard $\gamma$ as a representative of an element of either $H_1(S_g;\Z)$ or of $H_1(S_g,\{p_1,p_2\};\Z)$.  

Since $f \in \SBK(S_g,\{p_1,p_2\})$, it follows that $\phi$ is isotopic to the identity as a homeomorphism of $S_g$.  In particular, we have
\[ [\gamma] = [\phi(\gamma)] \in H_1(S_g;\Z).\]
There is a natural map $H_1(S_g;\Z) \to H_1(S_g,\{p_1,p_2\};\Z)$ where $[\gamma]$ maps to $[\gamma]$ and $[\phi(\gamma)]$ maps to $[\phi(\gamma)]$.  Since this map is well defined, it follows that
\[ [\gamma] = [\phi(\gamma)] \in H_1(S_g,\{p_1,p_2\};\Z),\]
which is what we wanted to show.
\end{proof}

Lemma~\ref{lem:fix2} tell us that in order to show that an element of the group $\SBK(S_g,\{p_1,p_2\})$ lies in the Torelli group, we only need to keep track of its action on the homology class of a single arc.  The only other ingredient we need in order to prove Theorem~\ref{thm:algchar} is a formula for how elements of $\SBK(S_g,\{p_1,p_2\})$ act on these classes.

Via the isomorphism $\pi_1(S_{0,2g+2},p) \to \SBK(S_g,\{p_1,p_2\})$, there is an action of $\pi_1(S_{0,2g+2},p)$ on $H_1(S_g,\{p_1,p_2\};\Z)$.  We denote the action of $\zeta \in \pi_1(S_{0,2g+2},p)$ by $\zeta_\star$.

Each generator $\zeta_i$ of $\pi_1(S_{0,2g+2},p)$ is represented by a simple loop in $S_{0,2g+2}$ based at $p$ (see Figure~\ref{figure:zetas}).  Each loop lies in the regular neighborhood of an arc in $S_{0,2g+2}$ that connects $p$ to the $i$th marked point of $S_{0,2g+2}$.  We denote the preimage in $(S_g,\{p_1,p_2\})$ of the $i$th such arc in $(S_{0,2g+2},p)$ by $\beta_i$.  We orient the $\beta_i$ so that they all emanate from the same marked point; see Figure~\ref{figure:zetadots}.

\begin{figure}[htb]
\psfrag{...}{$\dots$}
\psfrag{z2g-1}{$\beta_{2g+1}$}
\psfrag{z1}{$\beta_1$}
\psfrag{abar}{$p$}
\centerline{\includegraphics[scale=1.25]{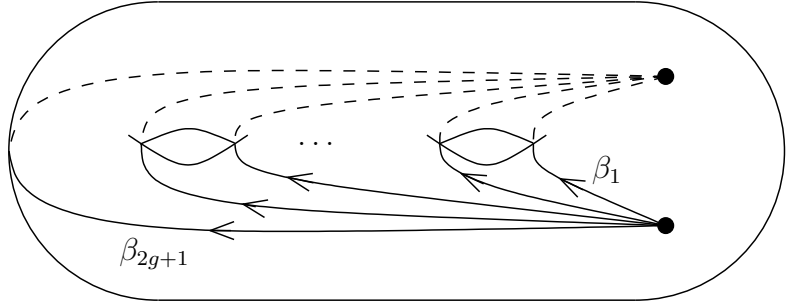}}
\caption{The arcs $\beta_i$ in $(S_g,\{p_1,p_2\})$.}
\label{figure:zetadots}
\end{figure}

For each $i$, we choose a neighborhood $N_i$ of $\beta_i$ that is fixed by $s$.  A \emph{half-twist} about $\beta_i$ is a homeomorphism of $(S_g,\{p_1,p_2\})$ that is the identity on the complement of $N_i$ and is described on $N_i$ by Figure~\ref{figure:zetaaction}.  This half-twist is well defined as a mapping class.

\begin{lem}
\label{lem:action}
Let $g \geq 1$, let $\zeta \in \pi_1(S_{0,2g+2},p)$, and say
\[ \zeta = \zeta_{i_1}^{\alpha_1} \cdots \zeta_{i_{m}}^{\alpha_m} \]
where $\zeta_{i_j} \in \{\zeta_i\}$ and $\alpha_i \in \{-1,1\}$.  We have the following formula for the action on $H_1(S_g,\{p_1,p_2\};\Z)$:
\[ \zeta_\star([\beta_k]) = [\beta_k] + 2\sum_{j=1}^{m} (-1)^j [\beta_{i_j}]. \]
\end{lem}

\begin{proof}

First of all, we claim that the image of $\zeta_i$ under the isomorphism $\pi_1(S_{0,2g+2},p) \to \SBK(S_g,\{p_1,p_2\})$ is the half-twist about $\beta_i$.  Indeed, the image of $\zeta_i$ under the map $\pi_1(S_{0,2g+2}) \to \Mod(S_{0,2g+2})$ is a Dehn twist about the boundary of a regular neighborhood of $\zeta_i$, and the unique lift of this Dehn twist to $\SBK(S_g,\{p_1,p_2\})$ is a half-twist about $\beta_i$.

We can now compute the action of $\pi_1(S_{0,2g+2},p)$ on $H_1(S_g,\{p_1,p_2\};\Z)$.  We first deal with the case where $\zeta = \zeta_i^{\pm 1}$.  If $i=k$, then we immediately see that the half-twist about $\beta_i$ (or its inverse) simply reverses the orientation of $\beta_k$, and so we have
\[ \zeta_\star([\beta_k]) = - [\beta_k] = [\beta_k] - 2[\beta_k] = [\beta_k]-2[\beta_i],\]
and the lemma is verified in this case.

If $\zeta=\zeta_i$ where $\zeta_i \neq \zeta_k$, then a neighborhood of $\beta_i \cup \beta_k$ in $(S_g,\{p_1,p_2\})$ is an annulus with two marked points.  As above, $\zeta=\zeta_i$ maps to the half-twist about $\beta_i$.  Simply by drawing the picture of the action (see Figure~\ref{figure:zetaaction}), we check the formula:
\[ \zeta_\star([\beta_k]) = [\beta_k] - 2[\beta_i]. \]
The case $\zeta=\zeta_i^{-1}$ is similar.

\begin{figure}[htb]
\psfrag{=}{$\cong$}
\psfrag{zi}{$\beta_i$}
\psfrag{zk}{$\beta_k$}
\psfrag{f}{$\zeta(\beta_k)$}
\centerline{\includegraphics[scale=1]{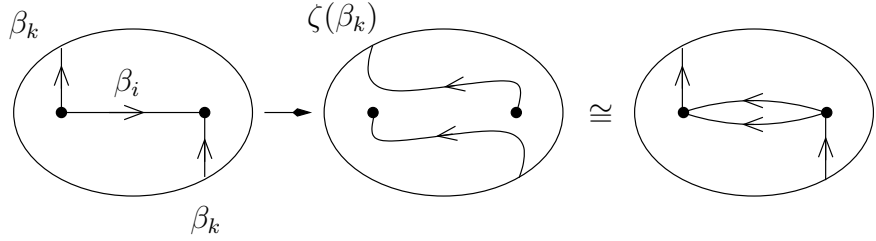}}
\caption{The action of  the half-twist about $\beta_i$ on $\beta_k$.}
\label{figure:zetaaction}
\end{figure}

Since the action of $\SBK(S_g,\{p_1,p_2\})$ on $H_1(S_g,\{p_1,p_2\};\Z)$ is linear, we can now complete the proof of the lemma by induction.  Suppose the lemma holds for $m-1$, that is, the induced action of $\zeta_{i_1}^{\alpha_i} \cdots \zeta_{i_{m-1}}^{\alpha_{m-1}}$ on $[\beta_k]$ is
\[ [\beta_k] \mapsto [\beta_k] + 2\sum_{j=1}^{m-1} (-1)^j [\beta_{i_j}]. \]
By linearity, and applying the case where $\zeta=\zeta_i^{\pm 1}$, the image of the latter homology class under $\zeta_{i_m}^{\alpha_m}$ is
\[ \left([\beta_k]-2[\beta_{i_m}]\right) + 2\sum_{j=1}^{m-1} (-1)^j \left([\beta_{i_j}]-2[\beta_{i_m}]\right), \]
which we rewrite as
\[ [\beta_k] + \left(2\sum_{j=1}^{m-1} (-1)^j [\beta_{i_j}]\right) + \left(4\sum_{j=1}^{m-1} (-1)^{j+1} [\beta_{i_m}]\right) -2[\beta_{i_m}]. \]
The sum of the third and fourth terms is $2(-1)^j[\beta_{i_m}]$, and so the lemma is proven.
\end{proof}

We are now poised to prove Theorem~\ref{thm:algchar}, which states that the map $\SI(S_g,\{p_1,p_2\}) \to \SI(S_g)$ is surjective and identifies its kernel with $\ker \epsilon$.

\begin{proof}[Proof of Theorem~\ref{thm:algchar}]

By Lemma~\ref{lem:fix2}, an element of $\SBK(S_g,\{p_1,p_2\})$ lies in $\SIBK(S_g,\{p_1,p_2\})$ if and only if it fixes the relative class $[\beta_1]$ in $H_1(S_g,\{p_1,p_2\};\Z)$.  It then follows from Lemma~\ref{lem:action} that an element of $\SBK(S_g,\{p_1,p_2\})$ fixes $[\beta_1]$ if and only if it lies in the image of $\ker \epsilon$.

It remains to show that there is a splitting $\SI(S_g) \to \SI(S_g,P)$.  By Theorem~\ref{thm:uncapping}, there is an injective homomorphism $\SI(S_g) \to \SI(S_g^1)$ with a left inverse induced by the inclusion $S_g^1 \to S_g$.   Via the (symmetric) inclusion $S_g^1 \to (S_g, \{ p_1,p_2 \})$, we obtain an injective homomorphism $\SI(S_g^1) \to \SI(S_g, \{ p_1,p_2 \})$.  The composition is the desired map $\SI(S_g) \to \SI(S_g, \{ p_1,p_2 \})$.  
%
%The group $\SIBK(S_g,\{p_1,p_2\})$ is isomorphic to the kernel of the surjective homomorphism $F_{2g+1}^{even} \to \Z^{2g+1}_{bal}$ (Theorem~\ref{thm:algchar} and Lemma~\ref{lem:im ep}), and hence is a free group of infinite rank.  This completes the proof.
%
\end{proof}

\subsection{Generating \boldmath$\SIBK(S_g,\{p_1,p_2\})$ by products of twists}

We will now use Theorem~\ref{thm:topchar} to give another description of the group $\SIBK(S_g,\{p_1,p_2\})$.

\begin{thm}
\label{thm:topchar}
For $g \geq 0$, each element of $\SIBK(S_g,\{p_1,p_2\})$ is a product of Dehn twists about symmetric separating simple closed curves.
\end{thm}

Theorem~\ref{thm:algchar} identifies $\SIBK(S_g,\{p_1,p_2\})$ with $\ker \epsilon$.  It is a general fact from combinatorial group theory that the kernel of a homomorphism is normally generated by elements that map to the defining relators for the image of the homomorphism.  We aim to exploit this fact, and so we start by determining the image of $\epsilon$.

Let $\Z^{2g+1} \to \Z$ be the map that records the sum of the coordinates, and let $\Z^{2g+1}_{bal}$ be the kernel.

\begin{lem}
\label{lem:im ep}
Let $g \geq 0$.  The image of $F_{2g+1}^{even}$ under $\epsilon$ is $\Z^{2g+1}_{bal}$.
\end{lem}

\begin{proof}

It follows immediately from the definition of the map $\epsilon$ that $\epsilon(F_{2g+1}^{even})$ lies in $\Z^{2g+1}_{bal}$.  To show that $\epsilon(F_{2g+1}^{even})$ is all of $\Z^{2g+1}_{bal}$, it suffices to show that $\Z^{2g+1}_{bal}$ is generated by the elements $\epsilon(\zeta_i\zeta_j)=e_i-e_j$, where $e_i$ is a generator the $i$th factor of $\Z^{2g+1}$.  We denote the element $e_i-e_j$ by $e_{i,j}$.

Let $\Z^{2g+1} \to \Z$ be the function that records the sum of the absolute values of the coordinates.  We think of this function as a height function.  The only element of $\Z^{2g+1}$ at height zero is the identity, which is the image of the identity element of $F_{2g+1}^{even}$.  Let $z$ be an arbitrary nontrivial element of $\Z^{2g+1}_{bal}$.  Since $z$ is nontrivial, it has at least one nonzero component, say the ${m}$th.  By the definition of $\Z^{2g+1}_{bal}$, there must be one component, say the $j$th, with opposite sign.  Say the ${m}$th component is negative and the $j$th component is positive.  The sum $\epsilon(\zeta_i\zeta_j)+z$ has height strictly smaller than that of $z$, so by induction the lemma is proven.
\end{proof}

\begin{lem}
\label{lem:zbalpres}
Let $g \geq 0$.  The group $\Z^{2g+1}_{bal}$ has a presentation:
\[ \langle e_{1,1}, \dots, e_{2g+1,2g+1} , e_{2,1}, \dots, e_{2g+1,1} \ |\  e_{i,i}=1, [e_{i,1},e_{j,1}]=1 \rangle. \]
\end{lem}

\begin{proof} 

Since $\Z^{2g+1}_{bal}$ is the subgroup of $\Z^{2g+1}$ described by one linear equation (the sum of the coordinates is 0), we see that $\Z^{2g+1}_{bal} \cong \Z^{2g}$.  Denote by $\eta$ the isomorphism given $\Z^{2g+1}_{bal} \to \Z^{2g}$ given by forgetting the first coordinate.

The group $\Z^{2g}$ is the free abelian group on $\eta(e_{2,1}), \dots, \eta(e_{2g+1,1})$, and so it has a presentation whose generators are $\eta(e_{2,1}), \dots, \eta(e_{2g+1,1})$ and whose relations are $[\eta(e_{i,1}),\eta(e_{j,1})]=1$.  

We thus obtain a presentation for $\Z^{2g+1}_{bal}$ with generators $e_{2,1}, \dots, e_{2g+1,1}$ and relations $[e_{i,1},e_{j,1}]=1$.  If we add (formal) generators $e_{i,i}$ to this presentation, as well as relations $e_{i,i}=1$, we obtain a new presentation for the same group; this is an elementary Tietze transformation \cite[Section 1.5]{mks}.
\end{proof}

We are now ready to prove Theorem~\ref{thm:topchar}.

\begin{proof}[Proof of Theorem~\ref{thm:topchar}]

Since $\SI(S^2,\{p_1,p_2\})=1$, we may assume $g \geq 1$.  By Lemma~\ref{lem:im ep}, we have a short exact sequence:
\[
1 \to \ker \epsilon \to F_{2g+1}^{even} \stackrel{\epsilon}{\to} \Z^{2g+1}_{bal} \to 1
\] 
where $\epsilon(\zeta_i \zeta_j) = e_{i,1}$ and $\epsilon(\zeta_i^2) = e_{i,i}=0$.

Consider the presentation for $\Z^{2g+1}_{bal}$ given in Lemma~\ref{lem:zbalpres}.  If we lift each relator in this presentation to an element of $F_{2g+1}^{even}$, we obtain a normal generating set for $\ker \epsilon$, that is, these elements and their conjugates in $F_{2g+1}^{even}$ generate $\ker \epsilon$.  The relators $e_{i,i}$ and $[e_{i,1},e_{j,1}]$ lift to elements
\[ \zeta_i^2 \quad \mbox{and} \quad [\zeta_i\zeta_1,\zeta_j\zeta_1],\]
respectively.

Passing through the isomorphism $F_{2g+1} \to \SBK(S_g,\{p_1,p_2\})$ from Theorem~\ref{thm:sbes2}, and applying Theorem~\ref{thm:algchar} we obtain a normal generating set for $\SIBK(S_g,\{p_1,p_2\})$.

Since the group generated by Dehn twists about symmetric separating curves is normal in $\SMod(S_g,\{p_1,p_2\})$ (hence in $\SBK(S_g,\{p_1,p_2\})$), it remains to show that the image of each $\zeta_i^2$ and $[\zeta_i\zeta_1,\zeta_j\zeta_1]$ in the group $\SBK(S_g,\{p_1,p_2\})$ can be written as a product of Dehn twists about symmetric separating curves.

To further simplify matters, the image of each $\zeta_i^2$ in $\SBK(S_g,\{p_1,p_2\})$ is conjugate to $\zeta_1^2$ in $\SMod(S_g,\{p_1,p_2\})$, and (up to taking inverses) the image of each $[\zeta_i\zeta_1,\zeta_j\zeta_1]$ is conjugate to $[\zeta_3\zeta_1,\zeta_2\zeta_1]$ in $\SMod(S_g,\{p_1,p_2\})$ (the point is that there are elements of $\Mod(S_{0,2g+2}, p)$ taking the elements $\zeta_1^2$ and $[\zeta_3\zeta_1,\zeta_2\zeta_1]$ of $\pi_1(S_{0,2g+2},p)$ to the other given elements).  Thus, we are reduced to checking that the images in $\SBK(S_g,\{p_1,p_2\})$ of $\zeta_1^2$ and $[\zeta_3\zeta_1,\zeta_2\zeta_1]$ are both products of Dehn twists about symmetric separating curves.

In the proof of Lemma~\ref{lem:action}, we showed that the image of $\zeta_1$ in the group $\SBK(S_g,\{p_1,p_2\})$ is a half-twist about the arc $\beta_1$.  It follows that the image of $\zeta_1^2$ is the Dehn twist about the boundary of a regular neighborhood of $\beta_1$.  This boundary is (isotopic to) a symmetric separating curve in $(S_g,\{p_1,p_2\})$.

It remains to analyze the element $[\zeta_3\zeta_1,\zeta_2\zeta_1]$.  There is a closed disk in $S_{0,2g+2}$ that contains the distinguished marked point $p$, the 1st, 2nd, and 3rd marked points of $S_{0,2g+2}$, and a representative of $[\zeta_3\zeta_1,\zeta_2\zeta_1] \in \pi_1(S_{0,2g+2},p)$.  Under the isomorphism $F_{2g+1} \to \SBK(S_g,\{p_1,p_2\})$ from Theorem~\ref{thm:sbes2}, we see that the commutator $[\zeta_3\zeta_1,\zeta_2\zeta_1]$ maps to an element of $\SI(S_g,\{p_1,p_2\})$ supported on a copy of $(S_1^1,\{p_1,p_2\})$ fixed by $s$.  Proposition~\ref{prop:g12} below states that $\SI(S_1^1,\{p_1,p_2\})$ is generated by Dehn twists about symmetric separating curves.  Combining this with the fact that the inclusion $(S_1^1,\{p_1,p_2\}) \to (S_g,\{p_1,p_2\})$ takes symmetric separating curves to symmetric separating curves, we conclude that $[\zeta_3\zeta_1,\zeta_2\zeta_1]$ is a product of Dehn twists about separating curves, and we are done.
\end{proof}

In the proof of Theorem~\ref{thm:topchar}, we used the following fact.

\begin{prop}
\label{prop:g12}
The group $\SI(S_1^1,\{p_1,p_2\})$ is generated by Dehn twists about symmetric separating curves.
\end{prop}

\begin{proof}

By the discussion after Theorem~\ref{thm:uncapping}, we have
\[ \SI(S_1^1,\{p_1,p_2\}) \cong \SI(S_1,\{p_1,p_2\}) \times \Z, \]
where the $\Z$ factor is the Dehn twist about $\partial S_1^1$.  Therefore, it suffices to show that $\SI(S_1,\{p_1,p_2\})$ is generated by Dehn twists about symmetric separating curves in $(S_1,\{p_1,p_2\})$.  This follows immediately from the fact that $\SMod(S_1,\{p_1,p_2\})=\Mod(S_1,\{p_1,p_2\})$ \cite[Section 3.4]{primer} and the fact that $\I(S_1,\{p_1,p_2\})$ is generated by Dehn twists about separating curves (this can be proven directly via the argument of \cite[Lemma 7.2]{bbm}, or it can be obtained immediately by combining \cite[Lemma 7.2]{bbm} with Lemma 5.8 below).
\end{proof}

%%%%
%%%%
%%%%

\section{Application to generating sets}
\label{sec:app}

Recall that Hain has conjectured that $\SI(S_g)$ is generated by Dehn twists about symmetric separating curves.  Since $\SI(S^2)$ and $\SI(T^2)$ are trivial, there is nothing to do for those cases.

\p{Genus two} Hain's conjecture is also known to be true in genus two.  Indeed, it follows from Fact~\ref{fact:low genus} that
\[ \SI(S_2) = \I(S_2).\]
A theorem of Birman and Powell gives that $\I(S_2)$ is generated by Dehn twists about separating curves \cite{birmansp,powell}.  It follows that Hain's conjecture is true for $\SI(S_2)$.  Applying Theorems~\ref{thm:sibes1} and~\ref{thm:uncapping}, we obtain the following extension. 

\begin{thm}
\label{thm:g21}
The groups $\SI(S_2)$, $\SI(S_2,p)$, and $\SI(S_2^1)$ are each generated by Dehn twists about symmetric separating curves.
\end{thm}

\p{Higher genus}  We now aim to apply Theorem~\ref{thm:topchar} in order to prove Theorem~\ref{thm:reducibles}, which states that, in order to prove Hain's conjecture, it is enough to show that $\SI(S_g)$ is generated by reducible elements.  To prove this theorem, we assume that $\SI(S_k)$ is generated by reducible elements for $k \leq g$, and we show that each reducible element of $\SI(S_g)$ is generated by Dehn twists about symmetric separating curves.

We say that an element of $\Mod(S_g)$ is \emph{strongly reducible} if it fixes the isotopy class of a simple closed curve in $S_g$.  We have the following theorem of Ivanov \cite[Theorem 3]{ivanov}.

\begin{thm}
\label{thm:pure}
Let $g \geq 0$.  If $f \in \I(S_g)$ is reducible, then $f$ is strongly reducible.
\end{thm}

We say that an isotopy class $a$ of simple closed curves is \emph{pre-symmetric} if it is not symmetric and $\sigma(a)$ and $a$ have disjoint representatives.

\subsection{Reduction to the symmetrically reducible case}
\label{section:reduction}

We say that an element $f$ of $\SMod(S_g)$ is \emph{symmetrically strongly reducible} if there is an isotopy class of simple closed curves in $S_g$ that is either symmetric or pre-symmetric and is fixed by $f$.

We have the following standard fact (see, e.g., \cite[Lemma 2.9]{primer}).  In the statement, we say that two simple closed curves $\alpha$ and $\beta$ are in minimal position if $|\alpha \cap \beta|$ is minimal with respect to the homotopy classes of $\alpha$ and $\beta$.

\begin{lem}
\label{lemma:alexander}
Let $S$ be any compact surface.  Let $\alpha$ and $\beta$ be two simple closed curves in $S$ that are in minimal position and that are not isotopic.  If $\phi : S \to S$ is a homeomorphism that preserves the set of isotopy classes $\{[\alpha],[\beta]\}$, then $\phi$ is isotopic to a homeomorphism that preserves the set $\alpha \cup \beta$.
\end{lem}

\begin{prop}
\label{prop:reduction}
Let $g \geq 0$.  If $f \in \SI(S_g)$ is strongly reducible, then $f$ is symmetrically strongly reducible.
\end{prop}

\begin{proof}

Let $a$ be an isotopy class of simple closed curves in $S_g$ that is fixed by $f$.  We may assume that $\sigma(a) \neq a$, for in that case there is nothing to do.
Since $f$ lies in $\SMod(S_g)$, we have:
\[ f(\sigma(a)) = \sigma(f(a)) = \sigma(a). \]
In other words, $f$ fixes the isotopy class $\sigma(a)$.  Since $\sigma$ has order 2, it  preserves the set of the isotopy classes $\{a,\sigma(a)\}$.

Let $\alpha$ and $\alpha'$ be representatives for $a$ and $\sigma(a)$ that are in minimal position.  Let $\mu$ denote the boundary of a closed regular neighborhood of $\alpha \cup \alpha'$, and let $\mu'$ denote the multicurve obtained from $\mu$ by deleting the inessential components of $\mu$ and replacing any set of parallel curves with a single curve.  Lemma~\ref{lemma:alexander} implies that both $f$ and $\sigma$ fix the isotopy class of $\mu'$.  By Theorem~\ref{thm:pure}, $f$ fixes the isotopy class of each component of $\mu'$.

Let $\mu_1, \dots, \mu_k$ denote the isotopy classes of the connected components of $\mu'$.   If $k=0$, that is to say $a$ and $\sigma(a)$ fill $S_g$, then it follows that $f$ has finite order (see \cite[Proposition 2.8]{primer}); hence $f$ is the identity.  Now suppose $k > 0$.  By construction, $i(\mu_i,\mu_j)=0$ for all $i$ and $j$, and $\sigma$ acts as an involution on the set of isotopy classes $\{[\mu_i]\}$.  Thus, there is either a singleton $\{[\mu_i]\}$ or a pair $\{[\mu_i],[\mu_j]\}$ fixed by $\sigma$, and the proposition is proven.
\end{proof}

\subsection{Analyzing individual stabilizers}
\label{subsection:cutting}

Let $a$ be the isotopy class of an essential simple closed curve in $S_g$.  Assume that $a$ is symmetric or pre-symmetric.  If $a$ is symmetric and separating, we choose a representative simple closed curve $\alpha$ so that ${s}(\alpha) = \alpha$, and if $a$ is nonseparating, we choose a representative simple closed curve $\alpha$ so that ${s}(\alpha) \cap \alpha = \emptyset$.  

Let $A$ denote either $\alpha$ or $\alpha \cup {s}(\alpha)$, depending on whether $a$ is separating or nonseparating, respectively.  Let $R_1$ and $R_2$ denote the closures in $S_g$ of the two connected components of $S_g-A$.  Let $R_1'$ and $R_2'$ denote the surfaces obtained from $R_1$ and $R_2$ obtained by collapsing each boundary component to a marked point.  Let $A'$ denote the set of marked points in either $R_1'$ or $R_2'$.

Each pair $(R_k',A')$ is homeomorphic to either $(S_g,p)$ where $p$ is a fixed point of ${s}$ or $(S_g,\{p_1,p_2\})$ where $p_1$ and $p_2$ are interchanged by ${s}$.  Since the hyperelliptic involution of $S_g$ induces a hyperelliptic involution of each $(R_k',A')$, we can define $\SMod(R_k',A')$ and $\SI(R_k',A')$ as in Section~\ref{sec:bg}.

We remark that if $a$ is symmetric and nonseparating, then one of the surfaces $R_k'$ is a sphere with two marked points.  For this surface, $\SMod(R_k',A') \cong \Z/2\Z$ and $\SI(R_k',A') = 1$.  For such $a$, it would have been more natural to take a representative $\alpha$ of $a$ that is symmetric.  However, the choice we made will allow us to make most of our arguments uniform for the various cases of $a$.

Let $\SMod(S_g,a)$ denote the stabilizer of the isotopy class $a$ in $\SMod(S_g)$, and let $\SMod(S_g,\vec a)$ denote the index 2 subgroup of $\SMod(S_g,a)$ consisting of elements that fix the orientation of $a$.  We now define maps
\[ \Psi_k : \SMod(S_g,\vec a) \to \SMod(R_k',A') \]
for $k=1,2$.

Let $f \in \SMod(S_g,\vec a)$, and let $\phi$ be a representative that commutes with $s$.  We may assume that $\phi$ fixes $\alpha$.  Since $\phi$ commutes with $s$, it must also fix $s(\alpha)$.  Since $f \in \SMod(S_g,\vec a)$, it follows that $\phi$ does not permute the components of $S_g-A$, and hence induces a homeomorphism $\phi_k'$ of  $R_k'$.  By construction, $\phi_k'$ commutes with the hyperelliptic involution of $R_k'$.  Finally, we define
\[ \Psi_k(f) = [\phi_k']. \]

We have the following standard fact; see \cite[Proposition 3.20]{primer}.

\begin{lem}
\label{lem:scut}
Let $g \geq 2$, and let $a$ be either a symmetric or pre-symmetric isotopy class of simple closed curves in $S_g$.  Define $R_i'$ and $A'$ as above.  The homomorphism
\[ \Psi_1 \times \Psi_2 : \SMod(S_g, \vec a) \to \SMod(R_1',A') \times \SMod(R_2',A') \]
is well defined and has kernel
\[ \ker(\Psi_1 \times \Psi_2) = \begin{cases} \langle T_a \rangle & a \ \ \text{symmetric} \\
\langle T_aT_{\sigma(a)} \rangle & a \ \ \text{pre-symmetric}  \end{cases} \]
\end{lem}

Let $\SI(S_g,a)$ denote $\SI(S_g) \cap \SMod(S_g,a)$.  Since $\SI(S_g,a)$ is a subgroup of $\SMod(S_g,\vec a)$, we can restrict each $\Psi_k$ to $\SI(S_g,a)$.

\begin{lem}
\label{lemma:cutting}
Let $g \geq 2$.  For $k \in \{1,2\}$, the image of $\SI(S_g,a)$ under $\Psi_k$ lies in $\SI(R_k',A')$.
\end{lem}

\begin{proof}

By the relative version of the Mayer--Vietoris sequence, we have an exact sequence:
\[ H_1(A,A) \to H_1(R_1,A) \times H_1(R_2,A) \to H_1(S_g,A) \to H_0(A,A). \]
(In this sequence, and in the rest of the proof, we take the coefficients for all homology groups to be $\Z$.)
The first and last groups are trivial, and so we have
\[  H_1(R_1,A) \times H_1(R_2,A) \cong H_1(S_g,A). \]
For each $k$, the map $R_k \to R_k'$ that collapses the boundary components to marked points induces an isomorphism
\[ H_1(R_k,A) \cong H_1(R_k',A'). \]

The natural map $H_1(S_g) \to H_1(S_g,A)$ is not surjective in general (it fails to be surjective in the case that $a$ is nonseparating).  However, the composition
\[ \pi : H_1(S_g) \to H_1(S_g,A) \stackrel{\cong}{\to} H_1(R_1',A') \times H_1(R_2',A') \to H_1(R_k',A') \]
is surjective for $k \in \{1,2\}$.  Indeed, any element $x$ of $H_1(R_1',A') \cong H_1(R_1,A)$ is represented by a collection of closed oriented curves in $R_1$ and oriented arcs in $R_1$ connecting $A$ to itself.  If we connect the endpoints of each oriented arc in $R_1$ by a similarly oriented arc in $R_2$, we obtain an element of $H_1(S_g)$ that maps to $x$.

By construction the following diagram is commutative:
\[
\xymatrix{
H_1(S_g) \ar[r]^{f_\star} \ar[d]_\pi & H_1(S_g) \ar[d]^\pi \\
H_1(R_k',A')\ar[r]^{\Psi_k(f)_\star} & H_1(R_k',A')
}
\]
and the lemma follows immediately.
\end{proof}

Let $\hat i(\cdot,\cdot)$ denote the algebraic intersection form on $H_1(S_g;\Z)$.
We will need the following lemma from \cite{companion}.

\begin{lem}
\label{lem:other paper}
Let $g \geq 2$, and let $a$ and $b$ be isotopy classes of oriented simple closed curves in $S_g$.  Suppose that $a$ is pre-symmetric, $b$ is symmetric, and $\hat i([a],[b])$ is odd.  Let $k \in \Z$.  If $[b] + k[a]$ is represented by a symmetric simple closed curve, then $k$ is even.
\end{lem}

\begin{lem}
\label{lem:image of pk}
Let $g \geq 2$, and let $a$ be the isotopy class of a simple closed curve in $S_g$ that is either symmetric or pre-symmetric.  Define $A'$ and $R_k'$ as above.   The homomorphism
\[ (\Psi_1 \times \Psi_2)|_{\SI(S_g,a)} : \SI(S_g,a) \to \SI(R_1',A') \times \SI(R_2',A') \]
is surjective with kernel
\[ \ker (\Psi_1 \times \Psi_2)|_{\SI(S_g,a)} = \begin{cases} \langle T_a \rangle & a \text{ is separating} \\ 1 & a \text{ is nonseparating}. \end{cases} \]
\end{lem}

\begin{proof}

The kernel of $(\Psi_1 \times \Psi_2)|_{\SI(S_g,a)}$ is $\ker(\Psi_1 \times \Psi_2) \cap \SI(S_g,a)$.  The description of $\ker (\Psi_1 \times \Psi_2)|_{\SI(S_g,a)}$ in the statement of the lemma then follows from Lemma~\ref{lem:scut}.

By Lemma~\ref{lemma:cutting}, we have
\[ (\Psi_1 \times \Psi_2)(\SI(S_g,a)) \subseteq \SI(R_1',A') \times \SI(R_2',A'). \]
It remains to show that $(\Psi_1 \times \Psi_2)|_{\SI(S_g,a)}$ is surjective.  Let $f' \in \SI(R_1',A') \times \SI(R_2',A')$.  Choose some $f \in \SMod(S_g,\vec a)$ that maps to $f'$.

Fix some orientation of $a$.  Consider the natural map
\[ \eta : H_1(S_g;\Z)/\langle [a] \rangle \to H_1(R_1',A';\Z) \times H_1(R_2',A';\Z).\]
The mapping classes $f$ and $f'$ induce automorphisms $f_\star$ and $f_\star'$ of $H_1(S_g;\Z)$ and $\textrm{Im}(\eta)$, respectively.  Since $f_\star([a])=[a]$, we further have that $f_\star$ induces an automorphism $\overline{f_\star}$ of $H_1(S_g;\Z)/\langle [a] \rangle$.

If we give $a$ an orientation, then it represents an element of $H_1(S_g;\Z)$.  There is a commutative diagram  
\[
\xymatrix{
H_1(S_g;\Z)/\langle [a] \rangle \ar[r]^{\overline{f_\star}} \ar[d]_\eta & H_1(S_g;\Z)/\langle [a] \rangle \ar[d]^\eta &\\
\textrm{Im}(\eta) \ar[r]^{f_\star'} & \textrm{Im}(\eta) & \hspace{-.5in} \subset \ \ H_1(R_1',A';\Z) \times H_1(R_2',A';\Z)
}
\]
Since $f_\star'$ is the identity and $\eta$ is injective, it follows that $\overline{f_\star}$ is the identity.

If $a$ is separating, then $[a]=0$ and so $f_\star = \overline{f_\star}$ is the identity.  Thus, $f$ is an element of $\SI(S_g,a)$, and since $\Psi_1 \times \Psi_2$ maps $f$ to $f'$, we are done in this case.

If $a$ is nonseparating, we can find an isotopy class $b$ of oriented symmetric simple closed curves in $S_g$ with $\hat i([a],[b])=1$.  Since $\overline{f_\star}$ is the identity, we have
\[ f_\star([b]) = [b] + k[a] \]
for some $k \in \Z$.

Our next goal is to find some $h \in \ker(\Psi_1 \times \Psi_2)$ so that $(hf)_\star$ fixes $[b]$.  If $a$ is symmetric, then we can simply take $h$ to be either $T_a^k$ (cf. \cite[Proposition 8.3]{primer}).  If $a$ is pre-symmetric, then this does not work, since $T_a^k \notin \SMod(S_g)$.  However, if $a$ is pre-symmetric, then Lemma~\ref{lem:other paper} implies that $k$ is even.  Thus we can take $h$ to be $\left(T_a T_{\sigma(a)}\right)^{k/2}$.

Now, let $x$ be any element of $H_1(S_g;\Z)$.  Since $\overline{f_\star}$ is the identity and since $h$ induces the identity map on $H_1(S_g;\Z)/\langle [a] \rangle$ we have $(hf)_\star(x) = x + j[a]$ for some $j \in \Z$.  Since $(hf)_\star$ is an automorphism of $H_1(S_g;\Z)$, we have:
\begin{eqnarray*}
\hat i(x,[b]) &=& \hat i ((hf)_\star(x),(hf)_\star([b])) \\
&=& \hat i (x+j[a],[b]) \\
&=& \hat i (x,[b]) + j\, \hat i([a],[b]) \\
&=& \hat i (x,[b]) + j
\end{eqnarray*}
and so $j=0$.  Thus $(hf)_\star(x)=x$ and so $hf \in \SI(S_g,a)$.  Since $h \in \ker(\Psi_1 \times \Psi_2)$, we have that $(\Psi_1 \times \Psi_2)(hf) = f'$, and we are done.
\end{proof}

\subsection{The inductive step}

In Section~\ref{section:reduction} we showed that reducible elements of $\SI(S_g)$ are strongly symmetrically reducible, and in Section~\ref{subsection:cutting} we studied strongly symmetrically reducible elements of $\SI(S_g)$.  We now combine the results from these sections with our Birman exact sequences for $\SI(S_g)$ in Section~\ref{section:bes} in order to prove Theorem~\ref{thm:reducibles}.

\begin{proof}[Proof of Theorem~\ref{thm:reducibles}]

Let $g \geq 2$.  As in the statement of the theorem, we assume that $\SI(S_k)$ is generated by reducible elements for all $k \in \{0,\dots,g\}$.  Fix some $k \in \{0,\dots,g\}$.  We make the inductive hypothesis that $\SI(S_j)$ is generated by Dehn twists about symmetric separating curves for $j \in \{0,\dots,k-1\}$, and we will show that $\SI(S_k)$ is generated by such Dehn twists.  The base cases $k=0,1$ are trivial since $\SI(S_0)=\SI(S_1)=1$.

Let $f \in \SI(S_k)$ be a reducible element.  By Theorem~\ref{thm:pure}, we have that $f$ is strongly reducible.  By Proposition~\ref{prop:reduction}, $f$ is symmetrically strongly reducible.  In other words, there is an isotopy class $a$ of essential simple closed curves in $S_k$, where $a$ is either symmetric or pre-symmetric, and where $f \in \SI(S_k,a)$.  

Define the subsets $A,R_1,R_2 \subset S_k$ as in Section~\ref{subsection:cutting}.  As per Lemma~\ref{lem:image of pk}, there is a (surjective) homomorphism
\[ (\Psi_1 \times \Psi_2)|_{\SI(S_k)} : \SI(S_k,a) \to \SI(R_1',A') \times \SI(R_2',A'), \]
and each element of the kernel is a power of a Dehn twist about a symmetric separating curve (when $a$ is nonseparating, the kernel is trivial).  For $i=1,2$, each Dehn twist about a symmetric separating simple closed curve in $\SI(R_i',A')$ has a preimage in $\SI(S_k,a)$ that is also a Dehn twist about a symmetric separating curve.  Thus, to prove the theorem, it suffices to show that each element of $\SI(R_i',A')$ is a product of Dehn twists about symmetric separating curves.

Fix $i \in \{1,2\}$.  Say that $R_i'$ has genus $g_i$.  Note that $0 \leq g_i < k$.  Combining Theorem~\ref{thm:sibes1} with Theorem~\ref{thm:topchar}, we have a short exact sequence
\[ 1 \to \SIBK(R_i',A') \to \SI(R_i',A') \to \SI(S_{g_i}) \to 1, \]
where each element of $\SIBK(R_i',A')$ is a product of Dehn twists about symmetric separating simple closed curves in $(R_i',A')$ (in the case that $a$ is separating, we have $\SIBK(R_i',A')=1$).  Our inductive hypothesis on $k$ gives that $\SI(S_{g_i})$ is generated by Dehn twists about symmetric separating curves.  Since each such Dehn twist has a preimage in $\SI(R_i',A')$ that is also a Dehn twist about a symmetric separating curve, it follows that each element of $\SI(R_i',A')$ is a product of Dehn twists about symmetric separating curves, and we are done.
\end{proof}

%%%%
%%%%
%%%%

\bibliographystyle{plain}
\bibliography{sibk}

\def\cprime{$'$} \def\cprime{$'$}
\begin{thebibliography}{10}

\bibitem{arnold}
V.~I. Arnol{\cprime}d.
\newblock A remark on the branching of hyperelliptic integrals as functions of
  the parameters.
\newblock {\em Funkcional. Anal. i Prilo\v zen.}, 2(3):1--3, 1968.

\bibitem{bbm}
Mladen Bestvina, Kai-Uwe Bux, and Dan Margalit.
\newblock The dimension of the {T}orelli group.
\newblock {\em J. Amer. Math. Soc.}, 23(1):61--105, 2010.

\bibitem{birmanes}
Joan~S. Birman.
\newblock Mapping class groups and their relationship to braid groups.
\newblock {\em Comm. Pure Appl. Math.}, 22:213--238, 1969.

\bibitem{birmansp}
Joan~S. Birman.
\newblock On {S}iegel's modular group.
\newblock {\em Math. Ann.}, 191:59--68, 1971.

\bibitem{birmanhilden}
Joan~S. Birman and Hugh~M. Hilden.
\newblock On the mapping class groups of closed surfaces as covering spaces.
\newblock In {\em Advances in the theory of Riemann surfaces (Proc. Conf.,
  Stony Brook, N.Y., 1969)}, pages 81--115. Ann. of Math. Studies, No. 66.
  Princeton Univ. Press, Princeton, N.J., 1971.

\bibitem{cdsi}
Tara Brendle, Leah Childers, and Dan Margalit.
\newblock {Cohomology of the hyperelliptic Torelli group}.
\newblock {2011}.
\newblock arXiv:1110.0448.

\bibitem{companion}
Tara Brendle and Dan Margalit.
\newblock {Symmetry and homology of surfaces}.
\newblock {In preparation}.

\bibitem{primer}
Benson Farb and Dan Margalit.
\newblock {\em {A primer on mapping class groups}}.
\newblock {Princeton University Press}, {2011}.

\bibitem{hain}
Richard Hain.
\newblock Finiteness and {T}orelli spaces.
\newblock In {\em Problems on mapping class groups and related topics},
  volume~74 of {\em Proc. Sympos. Pure Math.}, pages 57--70. Amer. Math. Soc.,
  Providence, RI, 2006.

\bibitem{ivanov}
Nikolai~V. Ivanov.
\newblock {\em Subgroups of {T}eichm\"uller modular groups}, volume 115 of {\em
  Translations of Mathematical Monographs}.
\newblock American Mathematical Society, Providence, RI, 1992.
\newblock Translated from the Russian by E. J. F. Primrose and revised by the
  author.

\bibitem{dj2}
Dennis Johnson.
\newblock The structure of the {T}orelli group. {II}. {A} characterization of
  the group generated by twists on bounding curves.
\newblock {\em Topology}, 24(2):113--126, 1985.

\bibitem{mks}
Wilhelm Magnus, Abraham Karrass, and Donald Solitar.
\newblock {\em Combinatorial group theory}.
\newblock Dover Publications Inc., New York, revised edition, 1976.
\newblock Presentations of groups in terms of generators and relations.

\bibitem{mp}
Wilhelm Magnus and Ada Peluso.
\newblock On a theorem of {V}. {I}. {A}rnol\cprime d.
\newblock {\em Comm. Pure Appl. Math.}, 22:683--692, 1969.

\bibitem{morifuji}
Takayuki Morifuji.
\newblock On {M}eyer's function of hyperelliptic mapping class groups.
\newblock {\em J. Math. Soc. Japan}, 55(1):117--129, 2003.

\bibitem{perron}
Bernard Perron.
\newblock Mapping class group and the {C}asson invariant.
\newblock {\em Ann. Inst. Fourier (Grenoble)}, 54(4):1107--1138, 2004.

\bibitem{powell}
Jerome Powell.
\newblock Two theorems on the mapping class group of a surface.
\newblock {\em Proc. Amer. Math. Soc.}, 68(3):347--350, 1978.

\bibitem{cutpaste}
Andrew Putman.
\newblock Cutting and pasting in the {T}orelli group.
\newblock {\em Geom. Topol.}, 11:829--865, 2007.

\bibitem{vdb}
Barbara van~den Berg.
\newblock {On the abelianization of the Torelli group}.
\newblock {\em Ph.D. thesis, University of Utrecht}, 2003.

\end{thebibliography}

\end{document}